\documentclass[leqno]{article}

\usepackage{amsmath}
\usepackage{amsthm}
\usepackage{mathrsfs}
\usepackage{amsfonts}
\begin{document}
\topmargin= -.2in \baselineskip=20pt

\newtheorem{theorem}{Theorem}[section]
\newtheorem{proposition}[theorem]{Proposition}
\newtheorem{lemma}[theorem]{Lemma}
\newtheorem{corollary}[theorem]{Corollary}
\newtheorem{conjecture}[theorem]{Conjecture}
\theoremstyle{remark}
\newtheorem{remark}[theorem]{Remark}

\title {Deformation of $\ell$-adic Sheaves with Undeformed Local Monodromy
\thanks{I would like to thank S. Bloch and H. Esnault for patiently explaining to me
their work \cite{BE}, which inspires this paper. In May 2005, H.
Esnault invited me to visit Essen and suggested me to work on the
problem studied in this paper. It is my great pleasure to thank her
for all the help. The research is partly supported by the NSFC
(10525107, 10921061) and 973 (2011CB935904).}}

\author {Lei Fu\\
{\small Chern Institute of Mathematics and LPMC, Nankai University,
Tianjin 300071, P. R. China}\\
{\small leifu@nankai.edu.cn}}

\date{}
\maketitle

\begin{abstract}
Let $X$ be a smooth connected algebraic curve over an algebraically
closed field $k$. We study the deformation of $\ell$-adic Galois
representations of the function field of $X$ while keeping the local
Galois representations at all places undeformed.

\noindent {\bf Mathematics Subject Classification:} 14D15.

\end{abstract}

\section*{Introduction}

In this paper, we work over an algebraically closed field $k$ of
characteristic $p$ even though our results can be extended to
non-algebraically closed fields. Let $X$ be a smooth connected
projective curve over $k$, let $S$ be a finite closed subset of $X$,
and let $\ell$ be a prime number distinct from $p$. For any $s\in
S$, let $\eta_s$ be the generic point of the strict henselization of
$X$ at $s$. A lisse $\overline{\mathbb Q}_\ell$-sheaf $\mathscr F$
on $X-S$ is called {\it physically rigid} if for any lisse
$\overline{\mathbb Q}_\ell$-sheaf $\mathscr G$ on $X-S$ with the
property $\mathscr F|_{\eta_s}\cong \mathscr G|_{\eta_s}$ for any
closed point $s$ in $S$, we have $\mathscr F\cong \mathscr G$. The
lisse $\overline{\mathbb Q}_\ell$-sheaf $\mathscr F$ on $X-S$
corresponds to a Galois representation
$$\rho:\mathrm{Gal}(\overline
{K(X)}/K(X))\to\mathrm{GL}(n,\overline{\mathbb Q}_\ell)$$ of the
function field $K(X)$ unramified everywhere on $X-S$. $\mathscr F$
is physically rigid if and only if for any Galois representation
$\rho'$ of $\mathrm{Gal}(\overline {K(X)}/K(X))$ such that $\rho'$
and $\rho$ induce isomorphic Galois representations of the local
field obtained by taking completion of $K(X)$ at any place of
$K(X)$, we have $\rho\cong \rho'$. In another words, a physically
rigid sheaf $\mathscr F$ is completely determined by all the Galois
representations of local fields defined by $\mathscr F$. To get a
good notion of rigidity, we have to assume $X=\mathbb P_k^1$.
Indeed, if $X$ has genus $g\geq 1$, then there exists a lisse
$\overline{\mathbb Q}_\ell$-sheaf $\mathscr L$ of rank $1$ on $X$
such that $\mathscr L^{\otimes n}$ is nontrivial for all $n$. For
any lisse $\overline{\mathbb Q}_\ell$-sheaf $\mathscr F$ on $X-S$,
the lisse sheaf $\mathscr G=\mathscr F\otimes\mathscr L$ is not
isomorphic to $\mathscr F$ since they have non-isomorphic
determinant, but $\mathscr F|_{\eta_s}\cong\mathscr G|_{\eta_s}$ for
all $s\in X$. Hence $\mathscr F$ is not rigid. A lisse
$\overline{\mathbb Q}_\ell$-sheaf $\mathscr F$ on $X-S$ is called
{\it cohomologically rigid} if we have
$$H^1(X, j_{\ast}\mathscr End(\mathscr F))=0,$$ where
$j:X-S\hookrightarrow X$ is the canonical open immersion. In
\cite[5.0.2]{K}, Katz shows that for an irreducible lisse sheaf,
cohomological rigidity implies physical rigidity. It is conjectured
that the converse is true.

In \cite[Theorem 4.10]{BE}, Bloch and Esnault study deformations of
locally free $\mathscr O_{X-S}$-modules provided with connections
while keeping local (formal) data undeformed, and they prove that
physical rigidity and cohomological rigidity are equivalent for
locally free $\mathscr O_{X-S}$-modules provided with connections.
Motivated by their results, in this paper, we study the deformation
of lisse $\ell$-adic sheaves while keeping the local monodromy
undeformed. More precisely,  Let $F$ be any one of the following
fields: a finite extension of the finite field $\mathbb F_\ell$ with
$\ell$ elements, an algebraic closure $\overline{\mathbb F}_\ell$ of
$\mathbb F_\ell$, a finite extension of the $\ell$-adic number field
$\mathbb Q_\ell$, or an algebraic closure $\overline{\mathbb
Q}_\ell$ of $\mathbb Q_\ell$. Let $\mathscr F$ be a lisse $F$-sheaf
on $X-S$. In this paper, we study the deformation of $\mathscr F$ so
that $\mathscr F|_{\eta_s}$ $(s\in S)$ remain undeformed.

Let $\eta$ be a generic point of $X$. We define an
$F$-representation of $\pi_1(X-S,\bar \eta)$ of rank $r$ to be a
homomorphism $\rho:\pi_1(X-S,\bar \eta)\to \mathrm{GL}(F^r)$ such
that the following conditions holds: If $F$ is a finite extension of
$\mathbb F_\ell$ or $\mathbb Q_\ell$, we require $\rho$ to be
continuous, where the topology on $\mathrm{GL}(F^r)$ is the discrete
topology if $F$ is a finite field, and is induced by the the
$\ell$-adic topology on $F$ if $F$ is a finite extension of $\mathbb
Q_\ell$; if $F$ an algebraic closure of $\mathbb F_\ell$ (resp.
$\mathbb Q_\ell$), we require the existence of a finite extension
$E$ of $\mathbb F_\ell$ (resp. $\mathbb Q_\ell)$ such that $\rho$
factors through a continuous homomorphism $\pi_1(X-S,\bar\eta)\to
\mathrm{GL}(E^r)$. Let $V=\mathscr F_{\bar\eta}$. Then the lisse
$F$-sheaf $\mathscr F$ on $X-S$ defines an $F$-representation
$$\rho_0:\pi_1(X-S,\bar \eta)\to \mathrm{GL}(V).$$ Fix an embedding
$\mathrm{Gal}(\bar\eta_s/\eta_s)\hookrightarrow \pi_1(X-S,\bar\eta)$
for each $s\in S$. Our problem can be interpreted as the deformation
of the representation $\rho_0$ so that
$\rho_0|_{\mathrm{Gal}(\bar\eta_s/\eta_s)}$ $(s\in S)$ remain
undeformed. Our treatment is similar to Mazur's theory of
deformation of Galois representations (\cite{M}).

Denote by $\mathcal C$ the category of Artinian local $F$-algebras
with residue field $F$. Morphisms in $\mathcal C$ are $F$-algebra
homomorphisms. Using the fact that the maximal ideal of an Artinian
local ring coincides with its nilpotent radical, one can check that
morphisms in $\mathcal C$ are necessarily local homomorphisms and
they induce the identity $\mathrm{id}_F$ on the residue field. If
$A$ is an object in $\mathcal C$, we denote by $\mathfrak m_A$ the
maximal ideal of $A$. A homomorphism $\rho:\pi_1(X-S,\bar\eta)\to
\mathrm{GL}(A^r)$ is called a representation if by regarding $A^r$
as a finite dimensional $F$-vector space, $\rho$ is an
$F$-representation of $\pi_1(X-S,\bar\eta)$.

Let $\rho_1,\rho_2: \pi_1(X-S,\bar\eta)\to \mathrm{GL}(A^r)$ be two
representations. For any subgroup $G$ of $\pi_1(X-S,\bar\eta)$, we
say $\rho_1|_G$ and $\rho_2|_G$ are {\it equivalent} if there exists
$P\in \mathrm{GL}(A^r)$ such that $P^{-1}\rho_1(g)P=\rho_2(g)$ for
all $g\in G$. We say $\rho_1|_G$ and $\rho_2|_G$ are {\it strictly
equivalent} if the above condition holds for some $P$ with the
property $P\equiv I\mod \mathfrak m_A$. We write
$\rho_1|_G\cong\rho_2|_G$ if they are equivalent.

Fix an $F$-representation $\rho_0: \pi_1(X-S,\bar\eta)\to
\mathrm{GL}(F^r)$. For any $A\in\mathrm{ob}\,\mathcal C$, define
$R(A)$ to be the set of strict equivalent classes of representations
$\rho:\pi_1(X-S,\bar\eta)\to\mathrm{GL}(A^r)$ such that
$\rho\equiv\rho_0\mod \mathfrak m_A$ and
$\rho|_{\mathrm{Gal}(\bar\eta_s/\eta_s)}\cong
\rho_0|_{\mathrm{Gal}(\bar\eta_s/\eta_s)}$ for all $s\in S$. Each
element in $R(A)$ is called a {\it deformation} of $\rho_0$ with
$\rho_0|_{\mathrm{Gal}(\bar\eta_s/\eta_s)}$ being undeformed. The
main result of this paper is the following theorem.

\begin{theorem} Assume all elements in the set
$\mathrm{End}_{F[\pi_1(X-S,\bar\eta)]}(F^r)$ are scalar
multiplications, where $F^r$ is considered as an
$F[\pi_1(X-S,\bar\eta)]$-module through the representation $\rho_0$.
(This condition holds if $\rho_0$ is absolutely irreducible by
Schur's lemma).

(i) The functor $R:\mathcal C\to(\mathrm{Sets})$ is
pro-representable, that is, there exist a complete noetherian local
$F$-algebra $R_{\mathrm{univ}}$ with residue field $F$ and a
homomorphism
$$\rho_{\mathrm{univ}}:\pi_1(X-S,\bar\eta)\to\mathrm{GL}(R^r_{\mathrm{univ}})$$
with the properties $\rho_{\mathrm{univ}}\equiv\rho_0\mod \mathfrak
m_{R_{\mathrm{univ}}}$ and
$\rho_{\mathrm{univ}}|_{\mathrm{Gal}(\bar\eta_s/\eta_s)}\cong
\rho_0|_{\mathrm{Gal}(\bar\eta_s/\eta_s)}$ for all $s\in S$, such
that the homomorphism
$\pi_1(X-S,\bar\eta)\to\mathrm{GL}((R_{\mathrm{univ}}/\mathfrak
m_{R_{\mathrm{univ}}}^m)^r)$ induced by $\rho_{\mathrm{univ}}$ are
representations for all positive integers $m$, and for any element
$\rho:\pi_1(X-S,\bar\eta)\to\mathrm{GL}(A^r)$ in $R(A)$, there
exists a unique local $F$-homomorphism $R_{\mathrm {univ}}\to A$
which brings $\rho_{\mathrm{univ}}$ to $\rho$. We call
$R_{\mathrm{univ}}$ the {\bf universal deformation ring} of $\rho_0$
and $\rho_{\mathrm{univ}}$ the {\bf universal deformation}.

(ii) Let $F[\epsilon]$ be the ring of dual numbers over $F$. The
tangent space $R(F[\epsilon])$ of the functor $R$ is isomorphic to
$H^1(X, j_\ast \mathscr End(\mathscr F))$, where
$j:X-S\hookrightarrow X$ is the open immersion, and $\mathscr F$ is
the lisse $F$-sheaf on $X$ corresponding to the representation
$\rho_0$.

(iii) $R_{\mathrm{univ}}$ is isomorphic to the formal power series
ring $F[[t_1,\ldots, t_m]]$ with $m=\mathrm{dim}_F H^1(X, j_\ast
\mathscr End(\mathscr F)).$

(iv) If we don't assume elements in
$\mathrm{End}_{F[\pi_1(X-S,\bar\eta)]}(F^r)$ are scalar
multiplications, then $F$ has a pro-representable hull, that is,
there exist a complete noetherian local $F$-algebra
$R_{\mathrm{univ}}$ with residue field $F$ and a homomorphism
$\rho_{\mathrm{univ}}:\pi_1(X-S,\bar\eta)\to
\mathrm{GL}(R_{\mathrm{univ}}^r)$ with the properties
$\rho_{\mathrm{univ}}\equiv\rho_0\mod \mathfrak
m_{R_{\mathrm{univ}}}$ and
$\rho_{\mathrm{univ}}|_{\mathrm{Gal}(\bar\eta_s/\eta_s)}\cong
\rho_0|_{\mathrm{Gal}(\bar\eta_s/\eta_s)}$ for all $s\in S$, such
that for any element $\rho$ in $R(A)$, there exists a (not
necessarily unique) local $F$-homomorphism $R_{\mathrm {univ}}\to A$
which brings $\rho_{\mathrm{univ}}$ to $\rho$.
\end{theorem}

We will prove (i), (ii) and (iv) in \S 1, and prove (iii) in \S 2.

\section{Proof of the main theorem}

In this section, we prove study the pro-representablity of the
functor $R$ using Schlessinger's criterion (\cite{S}). We start with
a series of lemmas.

\begin{lemma} Let $G$ be a group, let $F$ be a field, and let $V_1$ and $V_2$ be
$F[G]$-modules which are finite dimensional over $F$. Then for any
commutative $F$-algebra $A$, we have a canonical isomorphism
$$\mathrm{Hom}_{F[G]}(V_1,V_2)\otimes_F A\cong
\mathrm{Hom}_{A[G]}(V_1\otimes_FA, V_2\otimes_FA).$$
\end{lemma}

\begin{proof} Fix bases for the $F$-vector spaces
$V_1$ and $V_2$. They give rise to bases for the $A$-modules
$V_1\otimes_FA$ and $V_2\otimes_FA$. Write homomorphisms between
these $A$-modules in terms of matrices using these bases. Suppose
$T:V_1\otimes_F A\to V_2\otimes_F A$ is an $A[G]$-module
homomorphism. Then we can write
$$T=\sum_i a_i E_i$$ such that $a_i\in A$ are linearly independent over
$F$ and $E_i$ are matrices with entries in $F$. For any $g\in G$,
let $M_g$ and $N_g$ be the matrices of the action of $g$ on $V_1$
and $V_2$, respectively. Note that entries of $M_g$ and $N_g$ lie in
$F$. We have
$$\sum_i a_i(E_iM_g-N_gE_i)=0.$$ By the linear independence of $a_i$
over $F$, we have $E_iM_g=N_gE_i$ for all $i$ and all $g\in G$. So
$E_i$ define $F[G]$-module homomorphisms from $V_1$ and $V_2$. This
shows that the canonical map $$\mathrm{Hom}_{F[G]}(V_1,V_2)\otimes_F
A\to \mathrm{Hom}_{A[G]}(V_1\otimes_FA, V_2\otimes_FA)$$ is
surjective. The injectivity of this map follows from the fact that
$\mathrm{Hom}_{F[G]}(V_1,V_2)\otimes_F A$ (resp.
$\mathrm{Hom}_{A[G]}(V_1\otimes_FA, V_2\otimes_FA)$) can be
considered as a subspace of  $\mathrm{Hom}_F(V_1,V_2)\otimes_F A$
(resp. $\mathrm{Hom}_A(V_1\otimes_FA, V_2\otimes_FA)$), and that
$$\mathrm{Hom}_F(V_1,V_2)\otimes_F
A\cong \mathrm{Hom}_A(V_1\otimes_FA, V_2\otimes_FA).$$ (The last
isomorphism follows from the fact $V_1\cong F^r$
$(r=\mathrm{dim}_FV)$.)
\end{proof}

\begin{lemma} Let $A$ be an Artinian local ring, $F=A/\mathfrak m_A$
the residue field, $\rho:G\to\mathrm{GL}(A^r)$ a homomorphism, and
$\rho_0:G\to \mathrm{GL}(F^r)$ the homomorphism defined by $\rho$
modulo $\mathfrak m_A$. Regard $A^r$ (resp. $F^r$) as a module over
$A[G]$ (resp. $F[G]$) through the representation $\rho$ (resp.
$\rho_0$.) Suppose all elements in $\mathrm{End}_{F[G]}(F^r)$ are
scalar multiplications. Then all elements in
$\mathrm{End}_{A[G]}(A^r)$ are scalar multiplications.
\end{lemma}

\begin{proof} Recall that an epimorphism $A\to B$ of Artinian local
rings is called {\it small} if its kernel $\mathfrak a$ is a
principal ideal with the property $\mathfrak a\mathfrak m_A=0$. Any
epimorphism of Artinian local rings can be written as a composite of
a series of small epimorphisms. To prove the lemma, it suffices to
prove the following statement: If $\phi:A\to B$ is a small
extension, and all elements in $\mathrm{End}_{B[G]}(B^r)$ and all
elements in $\mathrm{End}_{F[G]}(F^r)$ are scalar multiplications,
then all elements in $\mathrm{End}_{A[G]}(A^r)$ are scalar
multiplications. When $\phi$ is an isomorphism, this is obvious.
Suppose $\phi$ is not an isomorphism. Let $t$ be a generator of the
kernel of $\phi$. Since $t\mathfrak m_A=0$, multiplication by $t$
induces a homomorphism
$$A/\mathfrak m_A\to tA$$ which is necessarily injective. So if
$ta_1=ta_2$, then $a_1\equiv a_2\mod \mathfrak m_A$. Let $P\in
\mathrm{End}_{A[G]}(A^r)$. Then $\phi(P)$ lies in
$\mathrm{End}_{B[G]}(B^r)$. By our assumption, $\phi(P)$ is scalar.
So there exist $a\in A$ and a matrix $\Delta$ with entries in $A$
such that
$$P=aI+t\Delta.$$
For any $g\in G$, from the fact that $\rho(g)P=P\rho(g)$, we get
$t\rho(g)\Delta=t\Delta\rho(g)$. By our previous discussion, this
implies that $\rho(g)\Delta\equiv \Delta\rho(g)\mod \mathfrak m_A$.
Hence modulo $\mathfrak m_A$, $\Delta$ defines an element in
$\mathrm{End}_{F[G]}(F^r)$. By our assumption, there exist $a'\in A$
and a matrix $\Delta'$ with entries in $\mathfrak m_A$ such that
$$\Delta=a'I+\Delta'.$$ As $t\mathfrak m_A=0$, we have
$$P=aI+t\Delta=aI+t(a'I+\Delta')=(a+ta')I.$$ Hence $P$ is scalar.
\end{proof}

From now on let $F$ be a finite extension of the finite field
$\mathbb F_\ell$ or $\mathbb Q_\ell$, or an algebraic closure of
such a field, and let $\mathcal C$ be the category of Artinian local
$F$-algebras with residue field $F$. Let $X$ be a connected smooth
projective curve over an algebraically closed field $k$ of
characteristic $p$ distinct from $\ell$, $S$ a finite closed subset
of $X$, $\eta$ the generic point of $X$, and
$\rho_0:\pi_1(X-S,\bar\eta)\to\mathrm{GL}(F^r)$ an
$F$-representation of $\pi_1(X-S,\bar\eta)$. Define the functor
$R:\mathcal C\to(\mathrm {Sets})$ as in Theorem 0.1. We will apply
\cite[Theorem 2.11]{S} to this functor. Let $\phi':A'\to A$ and
$\phi'':A''\to A$ be morphisms in $\mathcal C$, and consider the map
\begin{eqnarray}
R(A'\times_A A'')\to R(A')\times_{R(A)}R(A'').
\end{eqnarray}

\begin{lemma} Suppose $\phi'':A''\to A$ is surjective. Then the map
(1) is surjective.
\end{lemma}

\begin{proof}
Let $\rho':\pi_1(X-S, \bar\eta)\to\mathrm{GL}(A'^r)$ and
$\rho'':\pi_1(X-S, \bar\eta)\to\mathrm{GL}(A''^r)$ be elements in
$R(A')$ and $R(A'')$, respectively, such that they have the same
image in $R(A)$. Then there exists $P\in\mathrm{GL}(A^r)$ such that
$P\equiv I\mod \mathfrak m_A$ and
$$\phi'\circ \rho'=P^{-1}(\phi''\circ \rho'')P.$$ Here for convenience,
we denote the homomorphism $\mathrm{GL}(A'^r)\to\mathrm{GL}(A^r)$
(resp. $\mathrm{GL}(A''^r)\to\mathrm{GL}(A^r)$) induced by $\phi'$
(resp. $\phi''$) also by $\phi'$ (resp. $\phi''$). Since $\phi''$ is
surjective. There exists $P''\in \mathrm{GL}(A''^r)$ such that
$\phi''(P'')=P$ and $P''\equiv I\mod \mathfrak m_{A''}$. Replacing
$\rho''$ by $P''^{-1}\rho'' P''$, we may assume
$$\phi'\circ \rho'=\phi''\circ \rho''.$$ We can then define a
representation
$$\rho:\pi_1(X-S,\bar\eta)\to\mathrm{GL}\Big((A'\times_AA'')^r\Big)$$
such that $\rho$ is mapped to $\rho'$ and $\rho''$ under the two
projections $A'\times_A A''\to A'$ and $A'\times_A A''\to A'$,
respectively. It is clear that $\rho\equiv \rho_0\mod\mathfrak
m_{A'\times_A A''}.$ To prove the lemma, it remains to verify that
$$\rho|_{\mathrm{Gal}(\bar\eta_s/\eta_s)}\cong
\rho_0|_{\mathrm{Gal}(\bar\eta_s/\eta_s)}$$ for any $s\in S$.

There exist $P_s'\in\mathrm{GL}(A'^r)$ and
$P_s''\in\mathrm{GL}(A''^r)$ such that
$$\rho'|_{\mathrm{Gal}(\bar\eta_s/\eta_s)}
=P_s'^{-1}(\rho_0|_{\mathrm{Gal}(\bar\eta_s/\eta_s)})P_s',\quad
\rho''|_{\mathrm{Gal}(\bar\eta_s/\eta_s)}
=P_s''^{-1}(\rho_0|_{\mathrm{Gal}(\bar\eta_s/\eta_s)})P_s''.$$ Since
$\phi'\circ \rho'=\phi''\circ \rho''$, we have
$$\phi'(P_s')^{-1}(\rho_0|_{\mathrm{Gal}(\bar\eta_s/\eta_s)})\phi'(P_s')
=\phi''(P_s'')^{-1}(\rho_0|_{\mathrm{Gal}(\bar\eta_s/\eta_s)})\phi''(P_s'').$$
So $\phi'(P_s')\phi''(P_s'')^{-1}\in \mathrm{GL}(A^r)$ defines an
automorphism of the representation
$$\mathrm{Gal}(\bar\eta_s/\eta_s)\stackrel{\rho_0}
\to\mathrm{GL}(F^r)\hookrightarrow \mathrm{GL}(A^r)$$ obtained from
$\rho_0$ by scalar extension from $F$ to $A$. By Lemma 1.1, we have
\begin{eqnarray*}
\mathrm{End}_{A[\mathrm{Gal}(\bar\eta_s/\eta_s)]}(A^r)&\cong&
\mathrm{End}_{F[\mathrm{Gal}(\bar\eta_s/\eta_s)]}(F^r)\otimes_{F}A,\\
\mathrm{End}_{A''[\mathrm{Gal}(\bar\eta_s/\eta_s)]}(A''^r)&\cong&
\mathrm{End}_{F[\mathrm{Gal}(\bar\eta_s/\eta_s)]}(F^r)\otimes_{F}A'',
\end{eqnarray*}
where $A^r$ (resp. $A''^r$) is considered as an
$A[\mathrm{Gal}(\bar\eta_s/\eta_s)]$-module (resp.
$A''[\mathrm{Gal}(\bar\eta_s/\eta_s)]$-module) through the
representation $\rho_0|_{\mathrm{Gal}(\bar\eta_s/\eta_s)}$. Since
$\phi'':A''\to A$ is surjective, the canonical homomorphism
$$\mathrm{End}_{A''[\mathrm{Gal}(\bar\eta_s/\eta_s)]}(A''^r)\to
\mathrm{End}_{A[\mathrm{Gal}(\bar\eta_s/\eta_s)]}(A^r)$$ is
surjective. Using the fact that an endomorphism $Q$ in
$\mathrm{End}_{A[\mathrm{Gal}(\bar\eta_s/\eta_s)]}(A^r)$ is an
isomorphism if and only if $Q$ induces an isomorphism of
$(A/\mathfrak m_A)^r$, we see that
$$\mathrm{Aut}_{A''[\mathrm{Gal}(\bar\eta_s/\eta_s)]}(A''^r)\to
\mathrm{Aut}_{A[\mathrm{Gal}(\bar\eta_s/\eta_s)]}(A^r)$$ is
surjective. We have shown that $\phi'(P_s')\phi''(P_s'')^{-1}$ lies
in $\mathrm{Aut}_{A[\mathrm{Gal}(\bar\eta_s/\eta_s)]}(A^r)$. So
there exists $Q_s''\in
\mathrm{Aut}_{A''[\mathrm{Gal}(\bar\eta_s/\eta_s)]}(A''^r)$ such
that
$$\phi''(Q_s'')=\phi'(P_s')\phi''(P_s'')^{-1}.$$  We then have
\begin{eqnarray*}
&&\phi'(P_s')=\phi''(Q_s''P_s''),\\
&&\rho'|_{\mathrm{Gal}(\bar\eta_s/\eta_s)}
=P_s'^{-1}(\rho_0|_{\mathrm{Gal}(\bar\eta_s/\eta_s)})P_s',\quad
\rho''|_{\mathrm{Gal}(\bar\eta_s/\eta_s)}
=(Q''_sP''_s)^{-1}(\rho_0|_{\mathrm{Gal}(\bar\eta_s/\eta_s)})Q_s''P_s''.
\end{eqnarray*}
We can find $P_s\in\mathrm{GL}\Big((A'\times_A A'')^r\Big)$ which is
mapped to $P_s'$ and $Q_s''P_s''$ under the two projections
$A'\times_A A''\to A'$ and $A'\times_A A''\to A''$, respectively. We
then have $$\rho|_{\mathrm{Gal}(\bar\eta_s/\eta_s)}
=P_s^{-1}(\rho_0|_{\mathrm{Gal}(\bar\eta_s/\eta_s)})P_s.$$ This
finishes the proof of the lemma.
\end{proof}

\begin{lemma} Suppose $\phi'':A''\to A$ is surjective. If one of the following
conditions holds, then the map (1) is bijective.

(a) All elements in $\mathrm{End}_{F[\pi_1(X-S,\bar\eta)]}(F^r)$ are
scalar multiplications.

(b) $A=F$ and $A''=F[\epsilon]$ is the ring of dual numbers over
$F$.
\end{lemma}

\begin{proof} By Lemma 1.3, it suffices to show (1) is injective.
Let $\rho_1,\;\rho_2:\pi_1(X-S,\bar\eta)\to
\mathrm{GL}_{A'\times_A A''}\Big((A'\times_A A'')^r\Big)$ be two
elements in $R(A'\times_A A'')$ such that they have same images in
both $R(A')$ and $R(A'')$. Let $\psi':A'\times_AA''\to A'$ and
$\psi'':A'\times_AA''\to A''$ be the projections. Then there exist
$P'\in \mathrm{GL}(A'^r)$ and $P''\in\mathrm{GL}(A''^r)$ such that
\begin{eqnarray*}
P'\equiv I \mod \mathfrak m_{A'},&&\quad P''\equiv I \mod \mathfrak
m_{A''}, \\
\psi'(\rho_1)=P'^{-1}\psi'(\rho_2)P',&&\quad
\psi''(\rho_1)=P''^{-1}\psi''(\rho_2)P''.
\end{eqnarray*}
We then have
$$\phi'\psi'(\rho_1)=\phi'(P')^{-1}\phi'\psi'(\rho_2)\phi'(P'),\quad
\phi''\psi''(\rho_1)=\phi''(P'')^{-1}\phi''\psi''(\rho_2)\phi''(P'').$$
We have
$$\phi'\psi'(\rho_1)=\phi''\psi''(\rho_1),\quad
\phi'\psi'(\rho_2)=\phi''\psi''(\rho_2).$$ Set
$\rho=\phi'\psi'(\rho_2)=\phi''\psi''(\rho_2)$. Then we have
$$\Big(\phi'(P')\phi''(P'')^{-1}\Big)^{-1}\rho\Big(\phi'(P')\phi''(P'')^{-1}\Big)=\rho.$$

First we work under the condition (a). By Lemma 1.2,
$\phi'(P')\phi''(P'')^{-1}$ must be a scalar matrix. Choose a scalar
matrix $a''I$ such that
$$\phi'(P')\phi''(P'')^{-1}=\phi''(a'')I,$$
where $a''$ is a unit in $A''$ and $a''\equiv 1\mod\mathfrak
m_{A''}$. We have $\phi'(P')=\phi''(a''P'').$ So we can find $Q\in
\mathrm{GL}\Big((A'\times_AA'')^r\Big)$ such that
$$\psi'(Q)=P',\quad \psi''(Q)=a''P'',\quad Q\equiv I\mod \mathfrak
m_{A'\times_AA''}.$$ As
$$\psi'(\rho_1)=\psi'(Q)^{-1}\psi'(\rho_2)\psi'(Q),\quad
\psi''(\rho_1)=\psi''(Q)^{-1}\psi''(\rho_2)\psi''(Q),$$ we have
$\rho_1=Q^{-1}\rho_2 Q.$ So $\rho_1$ and $\rho_2$ are strictly
equivalent and they give rise to the same element in
$R(A'\times_AA'')$. Hence the map (1) is injective.

Next we work under the condition (b). Since $$P'\equiv I \mod
\mathfrak m_{A'},\quad P''\equiv I \mod \mathfrak m_{A''},$$ we have
$\phi'(P')=\phi''(P'')=I$. The above argument still works by taking
$a''=1$.
\end{proof}

Let $\phi':A'\to A$ be a homomorphism in the category $\mathcal C$
with kernel $\mathfrak a$, and let $G$ be a group. Two homomorphisms
$\rho_i:G\to\mathrm{GL}(A'^r)$ $(i=1,2)$ are called {\it strictly
equivalent relative to} $\phi'$ if $\rho_1\equiv \rho_2\mod
\mathfrak a$ and there exists $P\in\mathrm{GL}(A'^r)$ such that
$P\equiv I\mod \mathfrak a$ and $P^{-1}\rho_1P=\rho_2$.

\begin{lemma} Let $\phi':A'\to A$ be a
homomorphism in $\mathcal C$ with kernel $\mathfrak a$.

(i) Let $\rho_1,\rho_2:\pi_1(X-S,\bar\eta)\to \mathrm{GL}(A'^r)$ be
two representation such that $\rho_1\equiv \rho_2\mod \mathfrak a$
and such that $\rho_i\equiv \rho_0\mod \mathfrak m_{A'}$ $(i=1,2)$.
Suppose all elements in $\mathrm{End}_{F[\pi_1(X-S,
\bar\eta)]}(F^r)$ are scalar multiplications. Then $\rho_1$ is
equivalent to $\rho_2$ if and only if $\rho_1$ is strictly
equivalent to $\rho_2$ relative to $\phi'$.

(ii) Suppose furthermore that $\phi':A'\to A$ is surjective. Let
$\rho_1:\pi_1(X-S,\bar\eta)\to \mathrm{GL}(A^r)$ be a representation
such that $\rho_1\equiv \rho_0\mod \mathfrak a$. Then for any
subgroup $G$ of $\pi_1(X-S,\bar\eta)$, we have that $\rho_1|_G$ is
equivalent to $\rho_0|_G$ if and only if $\rho_1|_G$ is strictly
equivalent to $\rho_0|_G$ relative to $\phi'$.
\end{lemma}

\begin{proof} ${}$

(i) If $\rho_1$ is equivalent to $\rho_2$, then we can find
$P\in\mathrm{GL}(A'^r)$ such that $P^{-1}\rho_1P=\rho_2$. Modulo
$\mathfrak a$, this equation implies that $P_0^{-1}\rho P_0=\rho$,
where $P_0\in\mathrm{GL}(A^r)$ is the image of $P$ under the
homomorphism $\mathrm{GL}(A'^r)\to \mathrm{GL}(A^r)$, and
$\rho:\pi_1(X-S,\bar\eta)\to \mathrm{GL}(A^r)$ is the representation
such that $\rho=\rho_1\equiv \rho_2\mod \mathfrak a$. If all
elements in $\mathrm{End}_{F[\pi_1(X-S, \bar\eta)]}(F^r)$ are scalar
multiplications, then by Lemma 1.2, $P_0$ must be a scalar matrix.
Let $P'\in \mathrm{GL}(A'^r)$ be a scalar matrix lifting $P_0$. Then
we have $PP'^{-1}\equiv I\mod \mathfrak a$, and
$(PP'^{-1})^{-1}\rho_1(PP'^{-1})=\rho_2$. So $\rho_1$ is strictly
equivalent to $\rho_2$ relative to $\phi'$.

(ii) If $\rho_1|_G$ is equivalent to $\rho_0|_G$, then we can find
$P\in\mathrm{GL}(A'^r)$ such that $P^{-1}\rho_1(g)P=\rho_0(g)$ for
all $g\in G$. Modulo $\mathfrak a$, this equation implies that
$P_0^{-1}\rho_0(g) P_0=\rho_0(g)$ for all $g\in G$, where
$P_0\in\mathrm{GL}(A^r)$ is the image of $P$ under the homomorphism
$\mathrm{GL}(A'^r)\to \mathrm{GL}(A^r)$. So we have
$P_0\in\mathrm{Aut}_{A[G]}(A^r)$, where $A^r$ is regarded as a
$G$-module through the representation $\rho_0$. As in the proof of
Lemma 1.3, it follows from Lemma 1.1 that the canonical map
$$\mathrm{Aut}_{A[G]}(A'^r)\to \mathrm{Aut}_{A[G]}(A^r)$$ is
surjective if $A'\to A$ is surjective. Here $A'^r$ is regarded as a
$G$-module through the representation $\rho_0$. So we can find
$P'\in \mathrm{GL}(A'^r)$ such that $P'\equiv P_0\mod \mathfrak a$
and $P'^{-1}\rho_0(g)P'=\rho_0(g)$ for all $g\in G$. We have
$PP'^{-1}\equiv I\mod \mathfrak a$, and
$(PP'^{-1})^{-1}\rho_1(g)(PP'^{-1})=\rho_0(g)$ for all $g\in G$. So
$\rho_1|_G$ is strictly equivalent to $\rho_0|_G$ relative to
$\phi'$.
\end{proof}

Given an $F$-representation
$\pi_1(X-S,\bar\eta)\to\mathrm{GL}(F^r)$, we can talk about the
cohomology groups $H^i(\pi_1(X-S,\bar\eta),F^r)$. Indeed, for any
topological $\pi_1(X-S,\bar\eta)$-modules $M$, we can define a chain
complex $C^\cdot(\pi_1(X-S,\bar\eta),M)$ as in \cite[I 2.2]{Se} by
requiring $C^i(\pi_1(X-S,\bar\eta),M)$ to be the group of continuous
maps $\pi_1(X-S,\bar\eta)^i\to M$, and we define
$$H^i(\pi_1(X-S,\bar\eta),M)\cong H^i(C^\cdot(\pi_1(X-S,\bar\eta),M)).$$
This allows us to define $H^i(\pi_1(X-S,\bar\eta),F^r)$ in the case
where $F$ is a finite extension of $\mathbb F_\ell$ or $\mathbb
Q_\ell$. In the case where $F$ is a finite field, $F^r$ is a finite
discrete $\pi_1(X-S,\bar\eta)$-module. The cohomology groups of
discrete $\pi_1(X-S,\bar\eta)$-modules are studied in detail in
\cite{Se}. In the case where $F$ is a finite extension of $\mathbb
Q_\ell$, let $\Lambda$ be the integral closure of $\mathbb Z_\ell$
in $F$, let $\lambda$ be a uniformizer of $\Lambda$, and let $L$ be
a lattice in $F^r$ which is stable under the action of
$\pi_1(X-S,\bar\eta)$. Then we have
\begin{eqnarray*}
C^\cdot(\pi_1(X-S,\bar\eta),F^r)&\cong&
C^\cdot(\pi_1(X-S,\bar\eta),L)\otimes_\Lambda F,\\
H^i(\pi_1(X-S,\bar\eta),F^r)&\cong&
H^i(\pi_1(X-S,\bar\eta),L)\otimes_\Lambda F.
\end{eqnarray*}
Moreover, we have
$$C^\cdot(\pi_1(X-S,\bar\eta),L)\cong \varprojlim_n
C^\cdot(\pi_1(X-S,\bar\eta),L/\lambda^nL).$$ Using
\cite[$0_{\mathrm{III}}$ 13.2.3]{EGA}, one can show
$$H^i(\pi_1(X-S,\bar\eta),L)\cong \varprojlim_n
H^i(\pi_1(X-S,\bar\eta),L/\lambda^nL).$$ Note that $L/\lambda^nL$
are finite discrete $\pi_1(X-S,\bar\eta)$-modules. Finally in the
case where $F$ is an algebraic closure of $\mathbb F_\ell$ (resp.
$\mathbb Q_\ell$), choose a finite extension $E$ of $\mathbb F_\ell$
(resp. $\mathbb Q_\ell$) contained in $F$ such that the
representation $\pi_1(X-S,\bar\eta)\to\mathrm{GL}(F^r)$ is obtained
from a representation $\pi_1(X-S,\bar\eta)\to\mathrm{GL}(E^r)$ by
scalar extension. We define
$$H^i(\pi_1(X-S,\bar\eta),F^r)=H^i(\pi_1(X-S,\bar\eta),E^r)\otimes_EF.$$
Note that for any finite extension $E'$ of $E$ contained in $F$, we
have
\begin{eqnarray*}
C^i(\pi_1(X-S,\bar\eta),E'^r)&\cong&
C^i(\pi_1(X-S,\bar\eta),E^r)\otimes_EE',\\
H^i(\pi_1(X-S,\bar\eta),E'^r)&\cong&
H^i(\pi_1(X-S,\bar\eta),E^r)\otimes_EE',
\end{eqnarray*}
and $H^i(\pi_1(X-S,\bar\eta),F^r)$ is isomorphic to the $i$-th
cohomology group of the chain complex
$C'^\cdot(\pi_1(X-S,\bar\eta),F^r)$ defined by
$$C'^\cdot(\pi_1(X-S,\bar\eta),F^r)=\varinjlim_{E'}
C^\cdot(\pi_1(X-S,\bar\eta),E'^r),$$ where $E'$ goes over the set of
finite extensions of $E$ contained in $F$. Note that for this chain
complex, $C'^i(\pi_1(X-S,\bar\eta),F^r)$ is contained in the group
of all continuous maps $\pi_1(X-S,\bar\eta)^i\to F^r.$

If $M$ is a finite discrete $\pi_1(X-S,\bar\eta)$-module, then $M$
defines a locally constant etale sheaf $\mathscr M$ on $X-S$. If $M$
is a torsion discrete $\pi_1(X-S,\bar\eta)$-module, then $M$ is a
direct limit of finite discrete $\pi_1(X-S,\bar\eta)$-modules, and
hence $M$ also defines an etale sheaf $\mathscr M$ on $X-S$. By
\cite[XI 5]{SGA1}, we have
$$H^1(\pi_1(X-S,\bar\eta),M)\cong H^1(X-S,\mathscr M)$$ for any
finite discrete $\pi_1(X-S,\bar\eta)$-module $M$, and hence for any
torsion discrete $\pi_1(X-S,\bar\eta)$-module $M$ by \cite[VII
5.7]{SGA4} and \cite[Proposition I 8]{Se}. It follows from the above
the discussion that we have the following:

\begin{lemma} Given an $F$-representation
$\pi_1(X-S,\bar\eta)\to\mathrm{GL}(F^r)$, let $\mathscr F$ be the
corresponding $F$-sheaf on $X-S$. We have
$$H^1(\pi_1(X-S,\bar\eta),F^r)\cong H^1(X-S, \mathscr F).$$
\end{lemma}

\begin{lemma} Let $\phi':A'\to A$ be a nonzero homomorphism in the category $\mathcal C$
such that its kernel $\mathfrak a$ has the property $\mathfrak
m_{A'}\mathfrak a=0$. Then $\mathfrak a$ can be regarded as a vector
space over $F$ and $\mathfrak a^2=0$. Let
$\rho:\pi_1(X-S,\bar\eta)\to\mathrm{GL}(A^r)$ be a representation
such that $\rho\equiv \rho_0\mod \mathfrak m_A$ and such that $\rho$
can be lifted to a representation
$\pi_1(X-S,\bar\eta)\to\mathrm{GL}(A'^r)$. Then the set of strictly
equivalent classes relative to $\phi'$ of such representations
lifting $\rho$ can be identified with the set
$H^1(\pi_1(X-S,\bar\eta),\mathrm{Ad}(\rho_0))\otimes_F \mathfrak a$,
where $\mathrm{Ad}(\rho_0)$ is the $F$-vector space of $r\times r$
matrices with entries in $F$ on which $\pi_1(X-S,\bar \eta)$ acts by
the composition of $\rho_0$ with the adjoint representation of
$\mathrm{GL}(F^r)$.
\end{lemma}

\begin{proof} Fix a representation $\rho_1:\pi_1(X-S,\bar\eta)\to\mathrm{GL}(A'^r)$
lifting $\rho$. Any representation $\rho_2$ lifting $\rho$ can be
written in the form
$$\rho_2(g)=\rho_1(g)+\delta(g)\rho_1(g),$$ where $\delta(g)$ $(g\in
\pi_1(X-S,\bar\eta))$ are $r\times r$ matrices with entries in
$\mathfrak a$, and they define a continuous map
$\delta:\pi_1(X-S,\bar\eta)\to\mathrm{Ad}(\rho_0)\otimes_F\mathfrak
a$. Using the fact that $\rho_2(g_1g_2)=\rho_2(g_1)\rho_2(g_2)$, one
can verify $\delta$ is a 1-cocycle. Conversely, for any 1-cocycle
$\delta:\pi_1(X-S,\bar\eta)\to\mathrm{Ad}(\rho_0)\otimes_F\mathfrak
a$, the map $\rho_1+\delta\rho_1:\pi_1(X-S,\bar\eta)\to\mathrm
{GL}(A^r)$ is a representation lifting $\rho$. Suppose
$\delta,\delta':\pi_1(X-S,\bar\eta)\to\mathrm{Ad}(\rho_0)\otimes_F\mathfrak
a$ are two 1-cocycles such that $\delta-\delta'$ differs by a
1-coboundary, that is,
$$\delta'(g)-\delta(g)=\rho_1(g)M\rho_1(g)^{-1}-M$$ for some
$r\times r$ matrix $M$ with entries in $\mathfrak a$. Let
$\rho_2=\rho_1+\delta\rho_1$ and $\rho_2'=\rho_1+\delta'\rho_1$.
Then we have
$$\rho_2'=(I-M)\rho_2(I+M)=(I+M)^{-1}\rho_2'(I+M).$$
So $\rho_2$ and $\rho_2'$ are strictly equivalent relative to
$\phi'$. Conversely, if $\rho_2$ and $\rho_2'$ are two
representations lifting $\rho$ which are strictly equivalent
relative to $\phi'$, then the 1-cycles $\delta$ and $\delta'$
defined by $\rho_2=\rho_1+\delta\rho_1$ and
$\rho_2'=\rho_1+\delta'\rho_1$ differ by a 1-coboundary. This proves
our assertion.
\end{proof}

\begin{lemma} Let $j:X-S\hookrightarrow X$ be the canonical open immersion, and
let $\mathscr G$ be a lisse $F$-sheaf on $X-S$. Then we have a
canonical exact sequence
$$0\to H^1(X,j_\ast \mathscr G)\to H^1(X-S, \mathscr G)\to \bigoplus_{s\in
S}H^1(\eta_s, \mathscr G|_{\eta_s})\to H^2(X,j_\ast\mathscr G).$$
\end{lemma}

\begin{proof} Let $\Delta$ be the mapping cone of the canonical
morphism $j_\ast \mathscr G\to Rj_\ast \mathscr G$. We have a
distinguished triangle
$$j_\ast \mathscr G\to Rj_\ast \mathscr G
\to\Delta\to.$$ It gives rise to a long exact sequence
$$j_\ast\mathscr G\stackrel\cong\to
j_\ast\mathscr G\to\mathscr H^0(\Delta)\to 0\to R^1j_\ast \mathscr
G\to \mathscr H^1(\Delta)\to 0.$$ It follows that $\mathscr
H^i(\Delta)=0$ for $i\not=1$ and $\mathscr H^1(\Delta)\cong
R^1j_\ast \mathscr G$. Note that $R^1j_\ast\mathscr G$ is a
punctured sheaf supported on $S$, and for any $s\in S$, we have
$$(R^1j_\ast \mathscr G)_{\bar s}\cong H^1(\eta_s,
\mathscr G|_{\eta_s}).$$ It follows that
$$H^0(X,\Delta)=0,\quad H^1(X,\Delta)\cong\bigoplus_{s\in S}H^1(\eta_s,
\mathscr G|_{\eta_s}).$$ Taking the long exact sequence of
cohomology groups associated to the above distinguished triangle, we
get a long exact sequence
$$\begin{array}{cccccc}
H^0(X,\Delta)&\to H^1(X,j_\ast\mathscr G)\to & H^1(X,Rj_\ast
\mathscr G)&\to & H^1(X,\Delta)&\to
H^2(X,j_\ast\mathscr G).\\
\wr\!\parallel&&\wr\!\parallel&&\wr\!\parallel&\\
0&&H^1(X-S,\mathscr G)&& \bigoplus_{s\in S}H^1(\eta_s,\mathscr
G|_{\eta_s})&
\end{array}$$ Our assertion follows.
\end{proof}

\begin{lemma} We have a canonical isomorphism $R(F[\epsilon])\cong
H^1(X, j_\ast \mathscr End(\mathscr F))$, where
$j:X-S\hookrightarrow X$ is the open immersion, and $\mathscr F$ is
the lisse $F$-sheaf on $X$ corresponding to the representation
$\rho_0$.
\end{lemma}

\begin{proof}
By Lemma 1.7,  $H^1(\pi_1(X-S,\bar\eta),\mathrm{Ad}(\rho_0))$ can be
identified with the set of strict equivalent classes of
representations $\rho:\pi_1(X-S,\bar\eta)\to
\mathrm{GL}((F[\epsilon])^r)$ with the property
$\rho\equiv\rho_0\mod \epsilon$. Similarly,
$H^1(\mathrm{Gal}(\bar\eta_s/\eta_s),\mathrm{Ad}(\rho_0))$ can be
identified with the set of strict equivalent classes of
representations $\rho:\mathrm{Gal}(\bar\eta_s/\eta_s)\to
\mathrm{GL}((F[\epsilon])^r)$ with the property
$\rho\equiv\rho_0\mod \epsilon$.

Let's describe the kernel of the canonical homomorphism
$$H^1(\pi_1(X-S,\bar\eta),\mathrm{Ad}(\rho_0))\to
H^1(\mathrm{Gal}(\bar\eta_s/\eta_s),\mathrm{Ad}(\rho_0)).$$ Let
$\rho:\pi_1(X-S,\bar\eta)\to\mathrm{GL}((F[\epsilon])^r)$ be a
representation with the property $\rho\equiv\rho_0\mod \epsilon$.
The 1-cocyle $M:\pi_1(X-S,\bar\eta)\to
\mathrm{End}((F[\epsilon])^r)$ defined by $\rho=\rho_0+M\rho_0$
becomes a 1-coboundary with respect to the group
$\mathrm{Gal}(\bar\eta_s/\eta_s)$ if any only if
$\rho|_{\mathrm{Gal}(\bar\eta_s/\eta_s)}$ is strictly equivalent to
$\rho_0|_{\mathrm{Gal}(\bar\eta_s/\eta_s)}$. By Lemma 1.5 (ii), this
is equivalent to saying that
$\rho|_{\mathrm{Gal}(\bar\eta_s/\eta_s)}$ is equivalent to
$\rho_0|_{\mathrm{Gal}(\bar\eta_s/\eta_s)}$. Therefore the kernel of
the canonical homomorphism
$$H^1(\pi_1(X-S,\bar\eta),\mathrm{Ad}(\rho_0))\to
H^1(\mathrm{Gal}(\bar\eta_s/\eta_s),\mathrm{Ad}(\rho_0))$$ can be
identified with the set of strict equivalent classes of
representations
$\rho:\pi_1(X-S,\bar\eta)\to\mathrm{GL}((F[\epsilon])^r)$ such that
$\rho\equiv\rho_0\mod \epsilon$ and
$\rho|_{\mathrm{Gal}(\bar\eta_s/\eta_s)}$ is equivalent to
$\rho_0|_{\mathrm{Gal}(\bar\eta_s/\eta_s)}$.

By Lemma 1.8, $H^1(X,j_\ast\mathscr End(\mathscr F))$ can be
identified with the kernel of the canonical homomorphism
$$H^1(X-S,\mathscr End(\mathscr F))\to
\bigoplus_{s\in S}H^1(\eta_s,\mathscr End(\mathscr F)|_{\eta_s}).$$
By Lemma 1.6, we have canonical isomorphisms
$$H^1(X-S,\mathscr End(\mathscr F))\cong
H^1(\pi_1(X-S,\bar\eta),\mathrm{Ad}(\rho_0)),\quad H^1(\eta_s,
\mathscr End(\mathscr F)|_{\eta_s})\cong
H^1(\mathrm{Gal}(\bar\eta_s/\eta_s),\mathrm{Ad}(\rho_0)).$$ Combined
with the above discussion, we find that $H^1(X,j_\ast \mathscr
End(\mathscr F))$ can be canonically identified with the set of
strict equivalent classes of representations
$\rho:\pi_1(X-S,\bar\eta)\to\mathrm{GL}((F[\epsilon])^r)$ such that
$\rho\equiv\rho_0\mod \epsilon$ and
$\rho|_{\mathrm{Gal}(\bar\eta_s/\eta_s)}\cong
\rho_0|_{\mathrm{Gal}(\bar\eta_s/\eta_s)}$ for all $s\in S$. That
is, we have $H^1(X,j_\ast \mathscr End(\mathscr F))\cong
R(F[\epsilon])$.
\end{proof}

\begin{proof}[Proof of Theorem 0.1 (i), (ii), (iv)]
By Lemma 1.3, the condition $(H_1)$ in \cite[Theorem 2.11]{S} holds.
By Lemma 1.4 (b), the condition $(H_2)$ holds. By Lemma 1.9, the
tangent space of the functor $R$ can be identified with
$H^1(X,j_\ast\mathscr End(\mathscr F))$, which is finite dimensional
over $F$ by \cite[1.1]{Dfinitude}. So the condition $(H_3)$ holds.
Thus the functor $R$ has a pro-representable hull. Suppose
furthermore that $\mathrm{End}_{F[\pi_1(X-S,\bar\eta)]}(F^r)$
consists of scalar multiplications. Then the condition $(H_4)$ holds
by Lemma 1.4 (a). Thus the functor $R$ is pro-representable.
\end{proof}

\section{Obstruction to deformation}

Keep the notation of \S 1.

\begin{lemma} Suppose $S$ is nonempty.
For any torsion discrete $\pi_1(X-S,\bar\eta)$-module $M$, we have
$H^i(\pi_1(X-S,\bar\eta),M)=0$ for all $i\geq 2$.
\end{lemma}

\begin{proof} Let $I$ be a torsion discrete induced $\pi_1(X-S,\bar\eta)$-module
such that we have an embedding $M\hookrightarrow I$. We have
$H^i(\pi_1(X-S,\bar\eta),I)=0$ for all $i\geq 1$. (Confer \cite[I
2.5]{Se}.) So
$$H^i(\pi_1(X-S,\bar\eta),M)\cong
H^{i-1}(\pi_1(X-S,\bar\eta),I/M).$$ By induction on $i$, to prove
the lemma, it suffices to show $H^2(\pi_1(X-S,\bar\eta),M)=0$ for
any torsion discrete $\pi_1(X-S,\bar\eta)$-module $M$. As
$H^2(\pi_1(X-S,\bar\eta),I)=0$, it suffices to prove that the map
$$H^1(\pi_1(X-S,\bar\eta),I)\to H^1(\pi_1(X-S,\bar\eta),I/M)$$ is
surjective. Since $I$ and $M$ are torsion discrete
$\pi_1(X-S,\bar\eta)$-modules, they define $F$-sheaves $\mathscr I$
and $\mathscr M$ on $X-S$, respectively. By Lemma 1.6, we have
\begin{eqnarray*}
H^1(\pi_1(X-S,\bar\eta),I)&\cong& H^1(X-S,\mathscr I),\\
H^1(\pi_1(X-S,\bar\eta),I/M)&\cong& H^1(X-S,\mathscr I/\mathscr M).
\end{eqnarray*}
So it suffices to show that the map $$H^1(X-S,\mathscr I)\to
H^1(X-S,\mathscr I/\mathscr M)$$ the surjective. We have an exact
sequence
$$H^1(X-S,\mathscr I)\to
H^1(X-S,\mathscr I/\mathscr M)\to H^2(X-S,\mathscr M).$$ Since $S$
is nonempty, $X-S$ is an affine curve. So we have $H^2(X-S,\mathscr
M)=0$ by \cite[XIV 3.2]{SGA4}. Our assertion follows.
\end{proof}

Suppose $S$ is nonempty. Let $A'\to A$ be an epimorphism in the
category $\mathcal C$ such that its kernel $\mathfrak a$ has the
property $\mathfrak m_{A'}\mathfrak a=0$. Let
$\rho:\pi_1(X-S,\bar\eta)\to \mathrm{GL}(A^r)$ be a representation
such that $\rho\equiv \rho_0\mod \mathfrak m_A$. Fix a set theoretic
continuous lifting $\gamma:\pi_1(X-S,\bar\eta)\to \mathrm{GL}(A'^r)$
of $\rho$. Consider the map
\begin{eqnarray*}
c:\pi_1(X-S,\bar\eta)\times \pi_1(X-S,\bar\eta)&\to& \mathfrak
a\otimes_F \mathrm{End}(F^r)\cong \mathfrak a\otimes_F
\mathrm{Ad}(\rho_0),\\
c(g_1,g_2)&=&\gamma(g_1g_2)\gamma(g_2)^{-1}\gamma(g_1)^{-1}-1.
\end{eqnarray*}
One can show $c$ is a 2-cocycle. By Lemma 2.1, $c$ must be a
$2$-coboundary. Choose a continuous map
$$\delta:\pi_1(X-S,\bar\eta)\to \mathfrak a\otimes_k
\mathrm{Ad}(\rho_0)$$ such that $c=d(\delta\gamma^{-1})$. Then
$\rho'=\gamma+\delta:\pi_1(X-S,\bar\eta)\to \mathrm{GL}(A'^r)$ is a
representation lifting $\rho$. We have thus proved that $\rho$ can
always be lifted to a representation $\rho':\pi_1(X-S,\bar\eta)\to
\mathrm{GL}(A'^r)$.

Suppose furthermore that
$\rho|_{\mathrm{Gal}(\bar\eta_s/\eta_s)}\cong
\rho_0|_{\mathrm{Gal}(\bar\eta_s/\eta_s)}$ for any $s\in S$. Choose
$P_s\in \mathrm{GL}(A^r)$ such that $P_s^{-1}\rho(g)P_s=\rho_0(g)$
for all $g\in \mathrm{Gal}(\bar\eta_s/\eta_s)$. Choose $P_s'\in
\mathrm{GL}(A'^r)$ lifting $P_s$. Then
$(P_s'\rho_0P_s'^{-1})|_{\mathrm{Gal}(\bar\eta_s/\eta_s)}$ is a
lifting of $\rho|_{\mathrm{Gal}(\bar\eta_s/\eta_s)}$. Now
$\rho'|_{\mathrm{Gal}(\bar\eta_s/\eta_s)}$ is also a lifting of
$\rho|_{\mathrm{Gal}(\bar\eta_s/\eta_s)}$. As in the proof of Lemma
1.7, the continuous map $\delta_s:
\mathrm{Gal}(\bar\eta_s/\eta_s)\to
\mathrm{Ad}(\rho_0)\otimes_F\mathfrak a$ defined by
$$\rho'(g)=P_s'\rho_0(g)P_s'^{-1}+\delta_s(g) P_s'\rho_0(g)P_s'^{-1}\quad (g\in
\mathrm{Gal}(\bar\eta_s/\eta_s))$$ is a 1-cocycle. Let $[\delta_s]$
be the cohomology class of $\delta_s$ in
$H^1(\mathrm{Gal}(\bar\eta_s/\eta_s),\mathrm{Ad}(\rho_0)\otimes_F\mathfrak
a)$ and let $c$ be the image of $([\delta_s])_{s\in S}$ in the
cokernel of the canonical homomorphism
$$H^1(\pi_1(X-S,\bar\eta), \mathrm{Ad}(\rho_0)\otimes_F\mathfrak a)
\to \bigoplus_{s\in S}H^1(\mathrm{Gal}(\bar\eta_s/\eta_s),
\mathrm{Ad}(\rho_0)\otimes_F\mathfrak a).$$ By Lemma 1.8 and Lemma
1.6, this cokernel can be considered as a subspace of
$H^2(X,j_\ast\mathscr End(\mathscr F))\otimes_F\mathfrak a$. So we
can also regard $c$ as an element of in $H^2(X,j_\ast\mathscr
End(\mathscr F))\otimes_F\mathfrak a$. We call $c$ the {\it
obstruction class to lifting $\rho$ while preserving local data}.
For simplicity, in the sequel we simply call $c$ the obstruction
class to lifting $\rho$. In Lemma 2.2 below, we will show that $c$
is independent of the choice of $\rho'$, $P_s$ and $P_s'$. Note that
we have
\begin{eqnarray*}
\mathrm{det}(\rho'(g))&=&
\mathrm{det}\Big((I+\delta_s(g))P_s'\rho_0(g)P_s'^{-1}\Big)\\
&=& \big(1+\mathrm{Tr}(\delta_s(g))\big)\mathrm{det}(\rho_0(g))\\
&=&
\mathrm{det}(\rho_0(g))+\mathrm{Tr}(\delta_s(g))\mathrm{det}(\rho_0(g)).
\end{eqnarray*}
It follows that the obstruction class to lifting
$\mathrm{det}(\rho)$ is the image of the obstruction class to
lifting $\rho$ under the homomorphism
$$H^2(X,j_\ast\mathscr End(\mathscr F))\otimes_F\mathfrak a\to
H^2(X,F)\otimes_F\mathfrak a$$ induced by
$$\mathrm{Tr}:\mathscr End(\mathscr F)\to F.$$

\begin{lemma} Suppose $S$ is nonempty.
Let $\phi:A'\to A$ be an epimorphism in the category $\mathcal C$
such that its kernel $\mathfrak a$ has the property $\mathfrak
m_{A'}\mathfrak a=0$. Let $\rho:\pi_1(X-S,\bar\eta)\to
\mathrm{GL}(A^r)$ be a representation such that $\rho\equiv
\rho_0\mod \mathfrak m_A$ and
$\rho|_{\mathrm{Gal}(\bar\eta_s/\eta_s)}\cong
\rho_0|_{\mathrm{Gal}(\bar\eta_s/\eta_s)}$ for any $s\in S$. Let
$\rho':\pi_1(X-S,\bar\eta)\to \mathrm{GL}(A'^r)$ be a representation
lifting $\rho$ (which always exists) and define the obstruction
class $c$ to lifting $\rho$ as above.

(i) $c$ is independent of the choice of $\rho'$, $P_s$ and $P_s'$,
and $c$ vanishes if and only if $\rho$ can be lifted to a
representation $\rho'':\pi_1(X-S,\bar\eta)\to \mathrm{GL}(A'^r)$
such that $\rho''\equiv \rho_0\mod \mathfrak m_A$ and
$\rho''|_{\mathrm{Gal}(\bar\eta_s/\eta_s)}\cong
\rho_0|_{\mathrm{Gal}(\bar\eta_s/\eta_s)}$ for any $s\in S$.

(ii) The obstruction class to lifting $\mathrm{det}(\rho)$ is the
image of $c$ under the homomorphism
$$H^2(X,j_\ast\mathscr End(\mathscr F))\otimes_F\mathfrak a\to
H^2(X,F)\otimes_F\mathfrak a$$ induced by $\mathrm{Tr}:\mathscr
End(\mathscr F)\to F.$
\end{lemma}

\begin{proof} We have shown (ii) above. Let us prove (i).
Let $\rho'':\pi_1(X-S,\bar\eta)\to \mathrm{GL}(A'^r)$ be another
lifting of $\rho$, and define 1-cocycles
$$\delta_s,\theta_s:\mathrm{Gal}(\bar\eta_s/\eta_s)\to
\mathrm{Ad}(\rho_0)\otimes_F\mathfrak a$$ by
\begin{eqnarray*}
\rho'(g)&=&P_s'\rho_0(g)P_s'^{-1}+\delta_s(g)
P_s'\rho_0(g)P_s'^{-1},\\
\rho''(g)&=&P_s'\rho_0(g)P_s'^{-1}+\theta_s(g)
P_s'\rho_0(g)P_s'^{-1}
\end{eqnarray*}
for all $g\in \mathrm{Gal}(\bar\eta_s/\eta_s)$. Since $\rho'$ and
$\rho''$ are liftings of $\rho$, the continuous map
$$\psi:\pi_1(X-S,\bar\eta)\to\mathrm{Ad}(\rho_0)\otimes_F\mathfrak
a$$ defined by
$$\rho''(g)=\rho'(g)+\psi(g)\rho'(g)\quad
(g\in\pi_1(X-S,\bar\eta))$$ is a 1-cocycle for the group
$\pi_1(X-S,\bar\eta)$. For any $g\in
\mathrm{Gal}(\bar\eta_s/\eta_s)$, we have
\begin{eqnarray*}
(\theta_s(g)-\delta_s(g))P_s'\rho_0(g)P_s'^{-1}&=&\rho''(g)-\rho'(g)\\
&=& \psi(g)\rho'(g)\\
&=&\psi(g)(P_s'\rho_0(g)P_s'^{-1}+\delta_s(g)
P_s'\rho_0(g)P_s'^{-1})\\
&=&\psi(g)P_s'\rho_0(g)P_s'^{-1},
\end{eqnarray*}
where the last equality follows from the fact that $\mathfrak
a^2=0$. It follows that $$\theta_s(g)-\delta_s(g)=\psi(g)$$ for all
$g\in \mathrm{Gal}(\bar\eta_s/\eta_s)$. Hence the cohomology class
$[\theta_s]-[\delta_s]$ is the image of the cohomology class
$[\psi]$ under the canonical homomorphism
$$H^1(\pi_1(X-S,\bar\eta), \mathrm{Ad}(\rho_0)\otimes_F\mathfrak a)
\to H^1(\mathrm{Gal}(\bar\eta_s/\eta_s),
\mathrm{Ad}(\rho_0)\otimes_F\mathfrak a).$$ It follows that
$([\theta_s])_{s\in S}$ and $([\delta_s])_{s\in S}$ define the same
element in the cokernel of the canonical homomorphism
$$H^1(\pi_1(X-S,\bar\eta), \mathrm{Ad}(\rho_0)\otimes_F\mathfrak a)
\to \bigoplus_{s\in S}H^1(\mathrm{Gal}(\bar\eta_s/\eta_s),
\mathrm{Ad}(\rho_0)\otimes_F\mathfrak a).$$ So the obstruction class
to lifting $\rho$ is independent of the choice of the lifting
$\rho'$ of $\rho$.

Choose $\tilde P_s\in \mathrm{GL}(A^r)$ such that $\tilde
P_s^{-1}\rho(g)\tilde P_s=\rho_0(g)$ for all $g\in
\mathrm{Gal}(\bar\eta_s/\eta_s)$ and choose $\tilde P_s'\in
\mathrm{GL}(A'^r)$ lifting $\tilde P_s$. As representations of
$\mathrm{Gal}(\bar\eta_s/\eta_s)$, $\tilde P_s'\rho_0\tilde
P_s'^{-1}$ and $P_s'\rho_0P_s'^{-1}$ are equivalent and hence
strictly equivalent relative to $\phi$ by Lemma 1.5 (ii). Define
$$\delta_s'':\mathrm{Gal}(\bar\eta_s/\eta_s)\to\mathrm{Ad}(\rho_0)\otimes_F
\mathfrak a$$ by
$$\tilde
P_s'\rho_0(g)\tilde
P_s'^{-1}=P_s'\rho_0(g)P_s'^{-1}+\delta_s''(g)P_s'\rho_0(g)P_s'^{-1}$$
for all $g\in \mathrm{Gal}(\bar\eta_s/\eta_s)$. Then $\delta_s''$ is
a 1-coboundary. Define 1-cocycles
$$\delta_s,\tilde\delta_s:\mathrm{Gal}(\bar\eta_s/\eta_s)\to
\mathrm{Ad}(\rho_0)\otimes_F\mathfrak a$$ by
\begin{eqnarray*}
\rho'(g)&=&P_s'\rho_0(g)P_s'^{-1}+\delta_s(g)
P_s'\rho_0(g)P_s'^{-1},\\
\rho'(g)&=&\tilde P_s'\rho_0(g)\tilde P_s'^{-1}+\tilde\delta_s(g)
\tilde P_s'\rho_0(g)\tilde P_s'^{-1}
\end{eqnarray*}
for all $g\in \mathrm{Gal}(\bar\eta_s/\eta_s)$. Then we have
\begin{eqnarray*} \rho'(g)&=&\tilde P_s'\rho_0(g)\tilde
P_s'^{-1}+\tilde\delta_s(g)
\tilde P_s'\rho_0(g)\tilde P_s'^{-1}\\
&=&P_s'\rho_0(g)P_s'^{-1}+\delta_s''(g)P_s'\rho_0(g)P_s'^{-1}+\tilde
\delta_s(g)(P_s'\rho_0(g)P_s'^{-1}+\delta_s''(g)P_s'\rho_0(g)P_s'^{-1})
\\
&=&P_s'\rho_0(g)P_s'^{-1}+(\delta_s''(g)+
\tilde\delta_s(g))P_s'\rho_0(g)P_s'^{-1}.
\end{eqnarray*}
It follows that
$$\delta_s=\delta_s''+\tilde\delta_s$$ and hence $\delta_s$ and
$\tilde\delta_s$ differ by a 1-coboundary. So the obstruction class
to lifting $\rho$ is independent of the choice of $P_s$ and $P_s'$.

Suppose $\rho$ can be lifted to a representation
$\rho'':\pi_1(X-S,\bar\eta)\to \mathrm{GL}(A'^r)$ such that
$\rho''\equiv \rho_0\mod \mathfrak m_A$ and
$\rho''|_{\mathrm{Gal}(\bar\eta_s/\eta_s)}\cong
\rho_0|_{\mathrm{Gal}(\bar\eta_s/\eta_s)}$ for any $s\in S$. Then
$\rho''|_{\mathrm{Gal}(\bar\eta_s/\eta_s)}$ is equivalent to
$P_s'\rho_0|_{\mathrm{Gal}(\bar\eta_s/\eta_s)}P_s'^{-1}$. By Lemma
1.5 (ii), $\rho''|_{\mathrm{Gal}(\bar\eta_s/\eta_s)}$ and
$P_s'\rho_0|_{\mathrm{Gal}(\bar\eta_s/\eta_s)}P_s'^{-1}$ are
strictly equivalent relative to $\phi'$. By Lemma 1.7, the 1-cocycle
$\theta_s$ defined above becomes a 1-coboundary for the group
$\mathrm{Gal}(\bar\eta_s/\eta_s)$. By the above discussion, to
define the obstruction class to lifting $\rho$, we can use the
lifting $\rho''$ instead of the lifting $\rho'$. It follows that the
obstruction class vanishes.

Conversely, suppose the obstruction class $c$ to lifting $\rho$
vanishes. Then we can find a 1-cocycle
$\psi:\pi_1(X-S,\bar\eta)\to\mathrm{Ad}(\rho_0)\otimes_F\mathfrak a$
such that $\psi|_{\mathrm{Gal}(\bar\eta_s/\eta_s)}+\delta_s$ are
1-coboundaries for all $s\in S$. Set
$$\rho''=\rho'+\psi\rho'.$$ Then $\rho''$ is a lifting of $\rho$.
Moreover, for any $g\in \mathrm{Gal}(\bar\eta_s/\eta_s)$, we have
\begin{eqnarray*}
\rho''(g)&=&\rho'(g)+\psi(g)\rho'(g)\\
&=& P_s'\rho_0(g)P_s'^{-1}+\delta_s(g) P_s'\rho_0(g)P_s'^{-1} +
\psi(g)(P_s'\rho_0(g)P_s'^{-1}+\delta_s(g) P_s'\rho_0(g)P_s'^{-1})
\\
&=&
P_s'\rho_0(g)P_s'^{-1}+(\psi(g)+\delta_s(g))P_s'\rho_0(g)P_s'^{-1}.
\end{eqnarray*}
Since $\psi|_{\mathrm{Gal}(\bar\eta_s/\eta_s)}+\delta_s$ is a
1-coboundary for each $s\in S$,
$\rho''|_{\mathrm{Gal}(\bar\eta_s/\eta_s)}$ must be strictly
equivalent to $P_s'\rho_0(g)P_s'^{-1}$ relative to $\phi'$ by Lemma
1.7. In particular, $\rho''|_{\mathrm{Gal}(\bar\eta_s/\eta_s)}$ is
equivalent to $\rho_0|_{\mathrm{Gal}(\bar\eta_s/\eta_s)}$.
\end{proof}

\begin{lemma} Suppose $S$ is nonempty.
Let $\phi:A'\to A$ be an epimorphism in the category $\mathcal C$
such that its kernel $\mathfrak a$ has the property $\mathfrak
m_{A'}\mathfrak a=0$, and let $\rho:\pi_1(X-S,\bar\eta)\to
\mathrm{GL}(A^r)$ be a representation such that $\rho\equiv
\rho_0\mod \mathfrak m_A$ and
$\rho|_{\mathrm{Gal}(\bar\eta_s/\eta_s)}\cong
\rho_0|_{\mathrm{Gal}(\bar\eta_s/\eta_s)}$ for any $s\in S$. If the
rank $r$ of $\rho_0$ is $1$, then the obstruction class to lifting
$\rho$ vanishes.
\end{lemma}

\begin{proof} By Lemma 2.2 (i), it suffices to show that $R(A')\to
R(A)$ is surjective. For any $A\in\mathrm{ob}\,\mathscr C$, let
$R'(A)$ be the set of representations $\rho:\pi_1(X,\bar\eta)\to
A^\ast$ such that $\rho\equiv 1\mod \mathfrak m_A$. (Here we work
with representations of $\pi_1(X,\bar\eta)$, not those of
$\pi_1(X-S,\bar\eta)$. Recall that two rank $1$ representations are
equivalent if and only if they are equal.) Note that $R'(A)$ can be
identified with the set of rank $1$ representations $\rho$ of
$\pi_1(X-S,\bar \eta)$ such that $\rho\equiv 1\mod \mathfrak m_A$
and $\rho|_{\mathrm{Gal}(\bar\eta_s/\eta_s)}=1$ for all $s\in S$.
$R'$ is a functor from $\mathcal C$ to the category of sets, and we
have an isomorphism of functors $R'\stackrel\cong\to R$ defined by
the map
$$R'(A)\stackrel\cong\to R(A),\quad \rho\mapsto \rho\rho_0$$ for each
$A\in\mathrm{ob}\,\mathcal C$. Let us prove $R'(A')\to R'(A)$ is
surjective for any epimorphism $A'\to A$ in $\mathcal C$, that is,
the functor $R'$ is smooth. It suffices to prove the universal
deformation ring $R'_{\mathrm{univ}}$ for the trivial representation
$1:\pi_1(X,\bar \eta)\to F^\ast$ is a formal power series ring. Let
$A$ be an object in $\mathcal C$, and let $\rho:\pi_1(X,\bar\eta)\to
A^\ast$ be a representation such that $\rho\equiv 1\mod \mathfrak
m_A$. Then the image of $\rho$ is contained in the subgroup
$1+\mathfrak m_A$ of $A^\ast$. This subgroup has a filtration
$$1+\mathfrak m_A\supset 1+\mathfrak m_A^2\supset \cdots.$$ For each $i$,
we have an isomorphism of groups
$$\mathfrak m_A^i/\mathfrak m_A^{i+1}\cong
(1+\mathfrak m_A^i)/(1+\mathfrak m_A^{i+1}),$$ and $\mathfrak
m_A^i/\mathfrak m_A^{i+1}$ is the underlying abelian group of a
finite dimensional vector space over $F$. In the case where $F$ is a
finite extension of $\mathbb F_\ell$ or $\mathbb Q_\ell$, any
profinite subgroup of a finite dimensional $F$-vector space must be
a pro-$\ell$-group. It follows that the representation
$\rho:\pi_1(X,\bar\eta)\to A^\ast$ must factor through the
pro-$\ell$-completion of $\pi_1(X,\bar\eta)$ in this case. In the
case where $F=\overline {\mathbb F}_\ell$ or $\overline{\mathbb
Q}_\ell$, we can find a finite extension $E$ of $\mathbb F_\ell$ or
$\mathbb Q_\ell$ and a local Artinian $E$-algebra $A_E$ with residue
field $E$ such that $A_E\otimes_EF\cong A$ and such that
$\rho:\pi_1(X,\bar\eta)\to A^\ast$ factors through
$\pi_1(X,\bar\eta)\to A_E^\ast$. The above discussion shows that
$\rho:\pi_1(X,\bar\eta)\to A^\ast$ again factors through the
pro-$\ell$-completion of $\pi_1(X,\bar\eta)$. As the representation
is of rank $1$, $\rho$ factors through the abelianzation $\Gamma$ of
the pro-$\ell$-completion of $\pi_1(X,\bar\eta)$. By \cite[X
3.10]{SGA1}, the pro-$\ell$-completion of $\pi_1(X,\bar\eta)$ is
isomorphic to the pro-$\ell$-completion of the group with generators
$s_i, t_i$ $(1\leq i\leq g=\mathrm{genus}(X))$ and with one relation
$$(s_1t_1s_1^{-1}t_1^{-1})\cdots(s_gt_gs_g^{-1}t_g^{-1})=1.$$
It follows that $\Gamma\cong\mathbb Z_\ell^{2g}$. For any
nonnegative integer $m$,
$F[[t_1,\ldots,t_{2g}]]/(t_1,\ldots,t_{2g})^m$ is a finite
dimensional vector space over $F$, and hence has a natural topology
inherited from the topology on $F$. Endow $F[[t_1,\ldots,
t_{2g}]]=\varprojlim_m F[[t_1,\ldots,t_{2g}]]/(t_1,\ldots,t_{2g})^m$
with the projective limit topology. Then the homomorphism
$$\gamma:\mathbb Z^{2g}\to (F[[t_1,\ldots, t_{2g}]])^\ast,\quad (\lambda_1,\ldots,
\lambda_{2g})\to
(1+t_1)^{\lambda_1}\cdots(1+t_{2g})^{\lambda_{2g}}$$ is continuous
if we put the $\ell$-adic topology on $\mathbb Z^{2g}$. So it
induces a continuous homomorphism
$$\mathbb Z_\ell^{2g}\to (F[[t_1,\ldots, t_{2g}]])^\ast.$$
One can then verify the ring $F[[t_1,\ldots, t_{2g}]]$ together with
the representation $\pi_1(X,\bar\eta)\to (F[[t_1,\ldots,
t_{2g}]])^\ast$ defined by the composite
$$\pi_1(X,\bar\eta)\to\Gamma\cong \mathbb Z_\ell^{2g}\stackrel\gamma\to (F[[t_1,\ldots,
t_{2g}]])^\ast$$ satisfies the universal property required for the
universal deformation ring and the universal deformation. So the
universal deformation ring is isomorphic to a formal power series
ring. This finishes the proof of the lemma.
\end{proof}

\begin{proof}[Proof of Theorem 0.1 (iii)] Note that the pairing
$$\mathscr End(\mathscr F)\times \mathscr End(\mathscr F)\to
F,\quad (\phi,\psi)\mapsto \mathrm{Tr}(\psi\circ \phi)$$ defines a
self-duality on $\mathscr End(\mathscr F)$. By the duality theorem
(\cite[1.3 and 2.2]{Ddualite}), we have a perfect pairing
$$H^2(X,j_\ast\mathscr End(\mathscr F))\times H^0(X,j_\ast\mathscr
End(\mathscr F)(1))\to F.$$ If all elements in
$\mathrm{End}_{F[\pi_1(X-S,\bar\eta)]}(F^r)$ are scalar
multiplications, then we have
$$F\cong\mathrm{End}(\mathscr
F)\cong H^0(X,j_\ast\mathscr End(\mathscr F)).$$ So the morphism
$$F\to\mathscr End(\mathscr F),\quad a\mapsto aI$$ induces an
isomorphism
$$H^0(X,F)\cong H^0(X,j_\ast\mathscr End(\mathscr F)).$$
This implies that $\mathrm{Tr}:\mathscr End(\mathscr F)\to F$
induces an isomorphism
$$H^2(X,j_\ast\mathscr End(\mathscr F))\stackrel\cong\to H^2(X,F).$$
First consider the case where $S$ is nonempty. By Lemma 2.2, this
last isomorphism maps the obstruction class to lifting a deformation
of $\rho_0$ to the obstruction class to lifting a corresponding
deformation of $\mathrm{det}(\rho_0)$. By Lemma 2.3, there is no
obstruction to lifting a deformation of $\mathrm{det}(\rho_0)$. It
follows that there is no obstruction to lifting a deformation of
$\rho_0$. Hence the functor $R$ is smooth, and the universal
deformation ring $R_{\mathrm{univ}}$ is isomorphic to the formal
power series ring $F[[t_1,\ldots, t_m]]$ with $m=\mathrm{dim}_F
H^1(X, j_\ast \mathscr End(\mathscr F)).$

Next we consider the case where $S$ is empty. Fix a closed point
$\infty$ in $X$. For any $A\in\mathrm{ob}\,\mathcal C$, let $R'(A)$
be the set of strict equivalent classes of representations
$\rho:\pi_1(X-\{\infty\},\bar\eta)\to \mathrm{GL}(A^r)$ such that
$\rho\equiv\rho_0\mod\mathfrak m_A$ and
$\rho|_{\mathrm{Gal}(\bar\eta_{\infty}/\eta_\infty)}\cong
\rho_0|_{\mathrm{Gal}(\bar\eta_{\infty}/\eta_\infty)}$. Note that
$\rho_0|_{\mathrm{Gal}(\bar\eta_{\infty}/\eta_\infty)}$ is the
trivial representation of rank $r$, and the functor $R$ is
isomorphic to $R'$. We are thus reduced to the case where
$S=\{\infty\}$ is nonempty.
\end{proof}

\end{document}